\documentclass[reqno, 11pt]{amsart}

\usepackage[T1]{fontenc}        
\usepackage{ae,aecompl}
\usepackage{enumerate}
\usepackage{latexsym}
\usepackage{amsmath,amsthm,amsfonts}
\usepackage{amssymb}
\usepackage{epic}
\usepackage{tikz}
\usetikzlibrary{calc}
\usepackage{graphicx}
\usepackage{hyperref} 
\usepackage{url}
\usepackage{pgfplots}
\usepackage{subcaption}
\usepackage{algorithm}
\usepackage[noend]{algpseudocode}
\usepackage{float}
\usepackage{multicol}
\usetikzlibrary{matrix,shapes,decorations.pathreplacing}
\usepackage{easybmat}
\usepackage{multirow,bigdelim}
\makeatletter
\@namedef{subjclassname@2020}{%
  \textup{2020} Mathematics Subject Classification}
\makeatother

\usepackage[backend=biber]{biblatex} 
\DeclareFieldFormat{pages}{#1}
\title[Non-spherical sets versus lines in Euclidean Ramsey theory]{Non-spherical sets versus lines in Euclidean Ramsey theory}

\author[]{David Conlon}
\address{Department of Mathematics, California Institute of Technology, Pasadena, CA 91125, USA.}
\email{dconlon@caltech.edu}
\author[]{Jakob F\"uhrer}
\address{Institute of Analysis and Number Theory, Graz University of Technology,
Kopernikusgasse 24/II,
8010 Graz, Austria.}
\email{jakob.fuehrer@tugraz.at}

\date{}
\subjclass{05D10, 52C10}

\usepackage{geometry}

\usepackage{tikz-3dplot}
\usetikzlibrary{calc}

\begin{document}

 \maketitle
\newtheorem{theorem}{Theorem}
\newtheorem*{klar}{Klar}
\newtheorem{method}{Method}

\newtheorem{lemma}{Lemma}

\newtheorem{cor}{Corollary}

\newtheorem{conjecture}{Conjecture}

\theoremstyle{definition}
\newtheorem{defi}{Definition}

\newtheorem{bsp}{Beispiel}

\newtheorem*{bem}{Bemerkung}

\newtheorem*{vorschau}{Vorschau}

\newtheorem*{erg}{Ergänzung}

\theoremstyle{remark}
\newtheorem{remark}{Remark}

\newtheorem*{notation}{Notation}

\newtheorem{claim}{Claim}[theorem]
\renewcommand{\theclaim}{\arabic{claim}}
\newenvironment{proofofclaim}[1][\proofname\ of Claim \theclaim]{%
  \proof[#1]%
  \renewcommand\qedsymbol{$\blacksquare$}
}{\endproof}

\newcommand{\ndiv}{\not \hspace{3pt} \mid }

\newcommand*\hexbrace[2]{%
  \underset{#2}{\underbrace{\rule{#1}{0pt}}}}

\newcommand{\EE}{\mathbb{E}} 
\newcommand{\RR}{\mathbb{R}}

\begin{abstract}
We show that for every non-spherical set $X$ in $\mathbb{E}^d$, there exists a natural number $m$ and a red/blue-colouring of $\mathbb{E}^n$ for every $n$ such that there is no red copy of X and no blue progression of length $m$ with each consecutive point at distance $1$. This verifies a conjecture of Wu and the first author.
\end{abstract}


\section{Introduction}

Let $\mathbb{E}^n$ denote $n$-dimensional Euclidean space, that is, $\RR^n$ equipped with the Euclidean metric. Given finite sets $X_1, X_2, \dots, X_r \subset \EE^n$, we write $\mathbb{E}^n \rightarrow (X_1, X_2, \dots, X_r)$ if every $r$-colouring of $\EE^n$ contains a copy of $X_i$ in colour $i$ for some $i$, where a copy for us will always mean an isometric copy. Conversely, $\mathbb{E}^n \nrightarrow (X_1, X_2, \dots, X_r)$ means that there is some $r$-colouring of $\EE^n$ which does not contain a copy of $X_i$ in colour $i$ for any $i$. The Euclidean Ramsey problem, the study of which goes back to fundamental work of Erd\H{o}s, Graham, Montgomery, Rothschild, Spencer and Straus~\cite{EGMRSS1, EGMRSS2, EGMRSS3} in the 1970s, asks for a determination of those $X_1, X_2, \dots, X_r \subset \EE^n$ for which $\mathbb{E}^n \rightarrow (X_1, X_2, \dots, X_r)$.

In the particular case where $X_1 = X_2 = \dots = X_r = X$, we simply write $\mathbb{E}^n \rightarrow (X)_r$ to denote that every $r$-colouring of $\EE^n$ contains a monochromatic copy of $X$. Following Erd\H{o}s et al.~\cite{EGMRSS1}, we say that $X$ is \emph{Ramsey} if for every $r$ there exists $n$ such that $\mathbb{E}^n \rightarrow (X)_r$. The problem of determining those $X$ which are Ramsey is perhaps the most notorious question in this area and there are two rival conjectures for a characterisation. 

The first conjecture, already made by Erd\H{o}s et al.~in their first paper~\cite{EGMRSS1} on the subject, says that a finite set $X$ is Ramsey if and only if it is \emph{spherical}, meaning that it can be embedded in the surface of a sphere of some dimension. That being spherical is a necessary condition was already proved in~\cite{EGMRSS1} and subsequent results such as that of Frankl and R\"odl~\cite{FR} saying that all non-degenerate simplices are Ramsey and that of K\v r\'i\v z~\cite{K91} saying that regular polygons are Ramsey appear to add further weight.  

However, as pointed out by Leader, Russell and Walters~\cite{LRW}, all examples which are known to be Ramsey have the stronger property that they are \emph{subtransitive}, in the sense that they are subsets of finite sets which are transitive under the action of an appropriate group of isometries. This and other considerations then led them to make the rival conjecture that a finite set $X$ is Ramsey if and only if it is subtransitive. As there are finite sets which are spherical but not subtransitive (a non-obvious fact proved in~\cite{LRW2, LRW}), this is a strictly stronger conjecture, but, unlike the spherical sets conjecture, both directions of this conjecture remain open.

A result of Conlon and Fox~\cite{CF19} says that, under the axiom of choice, a finite set $X$ is Ramsey if and only if for every natural number $d$ and every fixed finite set $K \subset \mathbb{E}^d$, there exists $n$ such that $\mathbb{E}^n \rightarrow (X, K)$. That is, the problem of determining which sets are Ramsey can be recast in a somewhat simpler form. In~\cite{CW}, Conlon and Wu conjectured an even simpler characterisation, that a finite set $X$ is Ramsey if and only if for every natural number $m$, there exists $n$ such that $\mathbb{E}^n \rightarrow (X, \ell_m)$, where $\ell_m$ is the set consisting of $m$ points on a line with consecutive points at distance one. One direction of this conjecture, that if $X$ is Ramsey and $m$ is a natural number, then there exists $n$ such that $\mathbb{E}^n \rightarrow (X, \ell_m)$, follows from the result of Conlon and Fox (though the idea for this part of their result is essentially due to Szlam~\cite{Sz01}). However, the other direction remains open. Here we make some progress by proving the opposite direction for non-spherical sets. This verifies another conjecture made explicitly by Conlon and Wu~\cite{CW} and would settle their original conjecture in full if the spherical sets conjecture is true. 

\begin{theorem} \label{thm:main}
For every finite non-spherical set $X$, there exists a natural number $m$ such that $\mathbb{E}^n \nrightarrow (X, \ell_m)$ for all $n$.
\end{theorem}

The main result of~\cite{CW} was a proof of this conjecture in the particular case where $X$ is taken to be $\ell_3$, the simplest non-spherical set, which already answered a question raised independently by Conlon and Fox~\cite{CF19} and by Arman and Tsaturian~\cite{AT}. Their proof is probabilistic and shows that one may take $m \leq 10^{50}$. Through more explicit constructions, this bound has subsequently been improved, first by F\"uhrer and Toth~\cite{FT} to $m \leq 1177$ and then by Currier, Moore and Yip~\cite{CMY} to $m \le 20$. Both of these papers also proved certain further special cases of Theorem~\ref{thm:main}, though it remained wide open in full generality. Our construction here is again explicit, but the proof that it works makes use of some tools on equidistribution, namely, Weyl's equidistribution theorem and the Erd\H{o}s--Tur\'an--Koksma inequality.

\section{Proof of Theorem~\ref{thm:main}}

\subsection{The construction}

By a result of Erd\H{o}s et al~\cite[Lemma 14]{EGMRSS1}, there exist $c_1, \dots, c_s \in \RR$ and $B > 0$ such that every copy $\{x_1, \dots, x_s\}$ of the non-spherical configuration $X$ satisfies $\sum_{j=1}^s c_j |x_j|^2=B$. 
Without loss of generality, we can assume that $1\not\in \langle c_1,...,c_s \rangle_{\mathbb{Q}}$, as otherwise we can rescale the equation by a factor $\mu\not\in \langle c_1,...,c_s \rangle_{\mathbb{Q}}$.
Now let $b_1,...,b_r$ be a $\mathbb{Q}$-basis for $\langle c_1,...,c_s \rangle_{\mathbb{Q}}$ and let $q_{j,k}\in\mathbb{Q}$ be such that $c_j=\sum_{k=1}^rq_{j,k}b_k$. Let $M\in\mathbb{N}$ be such that $B':=MB>\sum_{j=1}^s\sum_{k=1}^r|q_{j,k}|$ and let $a_j:=Mb_j$. We may then recast the equation for copies of $X$ as 
\begin{equation}
\label{non-spherical_eq}
\sum_{j=1}^s\sum_{k=1}^rq_{j,k}a_k|x_j|^2=B'.
\end{equation}
We can also assume that all the $q_{j,k}$ are integral, as otherwise we can multiply the equation by their least common multiple. Let $p$ be a prime with $p>2B'$.
We now colour each point $x\in\mathbb{E}^n$ red if $\lfloor a_k|x|^2\rfloor \equiv 0$ (mod $p$) for all $k\in[1,r]$ and blue otherwise.

\subsection{No red copy of $X$}

Assume that $x_1,...,x_s$ are red points satisfying (\ref{non-spherical_eq}). Then 
$$\sum_{j=1}^s\sum_{k=1}^rq_{j,k}\lfloor a_k|x_j|^2\rfloor \equiv 0 \text{\;\;(mod $p$)}.$$ 
On the other hand,
\begin{equation*}
\left| \sum_{j=1}^s\sum_{k=1}^rq_{j,k}a_k|x_j|^2-\sum_{j=1}^s\sum_{k=1}^rq_{j,k}\lfloor a_k|x_j|^2\rfloor \right|< \sum_{j=1}^s\sum_{k=1}^r|q_{j,k}| <B'
\end{equation*} and therefore $$\sum_{j=1}^s\sum_{k=1}^rq_{j,k}\lfloor a_k|x_j|^2\rfloor\in (0,2B')\subseteq (0,p), $$ which is a contradiction.

\subsection{No blue copy of $\ell_m$}

Let $L = \{w_1, \dots, w_m\}$ be a copy of $\ell_m$. It was shown in~\cite[Section~3]{CW} that there exist $\beta, \gamma \in \RR$ depending on the choice of $L$ such that $y_j = |w_j|^2$ can be written in the form $y_j:=j^2+\beta j+\gamma$ for all $j = 1, \dots, m$. 
Consider the sequence 
$$Z:=(z_j)_{j\in[m]}:=\left(\left(\frac{a_1y_j}{p},...,\frac{a_ry_j}{p}\right)\right)_{j\in[m]}$$ in $(\mathbb{R}/\mathbb{Z})^r$. For $m$ sufficiently large and, crucially, independent of the choice of $\beta$ and $\gamma$, we will show that $\{z_j\}_{j\in[m]}\cap [0,1/p)^r\neq\emptyset$, which implies that there is no blue copy of $\ell_m$ in our construction.

Let $D_m(Z)$ be the discrepancy of $Z$ in $(\mathbb{R}/\mathbb{Z})^r$, the supremum over all axis-aligned boxes $B = \prod_{i=1}^r [a_i, b_i)$ of 
$$\left|\frac{A(B;Z)}{m} - \mu(B)\right|,$$
where $A(B; Z)$ counts the number of points of $Z$ in $B$ and $\mu(\cdot)$ is the Lebesgue measure on $(\mathbb{R}/\mathbb{Z})^r$. The key claim is as follows.

\begin{lemma} \label{disc}
$$D_m(Z)< \frac{1}{p^r}.$$ 
In particular, $\{z_j\}_{j\in[m]}\cap [0,1/p)^r\neq\emptyset$.
\end{lemma}

In order to prove this, we make use of the Erd\H{o}s--Tur\'an--Koksma inequality~\cite{Kok}.

\begin{lemma}[Erd\H{o}s--Tur\'an--Koksma] \label{ETK}
For every positive integer $N$,
$$D_m(Z)\leq C_r\left(\frac{1}{N} +\sum_{1 \leq ||h||_\infty \leq N} \frac{1}{c(h)}\left| \frac{1}{m} \sum_{j=1}^m e(\langle h,z_j\rangle)  \right|  \right),$$
where $c(h) = \prod_{i=1}^r \max\{1, |h_i|\}$ for $h = (h_1, \dots, h_r) \in \mathbb{Z}^r$ and $e(x)=\exp(2\pi i x)$.
\end{lemma}

To estimate the $\sum_{j=1}^m e(\langle h,z_j\rangle)$ term, we use the following special case of Weyl's equidistribution theorem~\cite[Satz 9]{Weyl}). 

\begin{lemma}[Weyl] \label{Weyl}
Let $P(x)= a x^2 + b x + c \in\mathbb{R}[x]$ 
be a quadratic polynomial with irrational leading coefficient $a$. 
Then
$$\left|\sum_{j=1}^m e(P(j))  \right|=o_{a}(m),$$ 
where the $o$ term does not depend on $b$ or $c$.
\end{lemma}

Note now that $\langle h,z_j\rangle$ describes a quadratic polynomial in $\mathbb{R}[j]$ with irrational leading coefficient $\sum_{k=1}^r h_{k}a_{k}$. Therefore, we have the following immediate corollary of Lemma~\ref{Weyl}.

\begin{cor}
\label{cor}
$$\left|\sum_{j=1}^m e(\langle h,z_j\rangle)  \right|=o_h(m).$$
\end{cor}



We are now in a position to prove Lemma~\ref{disc}.

\begin{proof}[Proof of Lemma~\ref{disc}]
Choose $N>2C_rp^r$ and then, using Corollary~\ref{cor}, $m$ such that $$C_r\sum_{1 \leq ||h||_\infty\leq N} \frac{1}{c(h)} \left| \frac{1}{m} \sum_{j=1}^m e(\langle h,z_j\rangle)  \right|<\frac{1}{2p^r}.$$
The result then follows from Lemma~\ref{ETK}.
\end{proof}

\subsection*{Acknowledgements}
 D.C. was supported by NSF Awards DMS-2054452 and DMS-2348859. 
 J.F. was supported by the Austrian Science Fund (FWF) under the project W1230.
 The authors also thank Manuel Hauke for helpful conversations regarding equidistribution.

\printbibliography

@article{EGMRSS1,
    AUTHOR = {Erd\H{o}s, P. and Graham, R. L. and Montgomery, P. and
              Rothschild, B. L. and Spencer, J. and Straus, E. G.},
     TITLE = {Euclidean {R}amsey theorems. {I}},
   JOURNAL = {J. Combinatorial Theory Ser. A},
  FJOURNAL = {Journal of Combinatorial Theory. Series A},
    VOLUME = {14},
      YEAR = {1973},
     PAGES = {341--363},
   MRCLASS = {05B30},
  MRNUMBER = {316277},
MRREVIEWER = {Vaclav\ Chv\'{a}tal},
%       DOI = {10.1016/0097-3165(73)90011-3},
%       URL = {https://doi.org/10.1016/0097-3165(73)90011-3},
}

@incollection {EGMRSS2,
    AUTHOR = {Erd\H{o}s, P. and Graham, R. L. and Montgomery, P. and
              Rothschild, B. L. and Spencer, J. and Straus, E. G.},
     TITLE = {Euclidean {R}amsey theorems. {II}},
 BOOKTITLE = {Infinite and finite sets ({C}olloq., {K}eszthely, 1973;
              dedicated to {P}. {E}rd\H{o}s on his 60th birthday), {V}ols.
              {I}, {II}, {III}},
    SERIES = {Colloq. Math. Soc. J\'{a}nos Bolyai},
    VOLUME = {10},
     PAGES = {529--557},
 PUBLISHER = {North-Holland, Amsterdam-London},
      YEAR = {1975},
   MRCLASS = {05C15},
  MRNUMBER = {382047},
MRREVIEWER = {Vaclav\ Chv\'{a}tal},
}

@incollection {EGMRSS3,
    AUTHOR = {Erd\H{o}s, P. and Graham, R. L. and Montgomery, P. and
              Rothschild, B. L. and Spencer, J. and Straus, E. G.},
     TITLE = {Euclidean {R}amsey theorems. {III}},
 BOOKTITLE = {Infinite and finite sets ({C}olloq., {K}eszthely, 1973;
              dedicated to {P}. {E}rd\H{o}s on his 60th birthday), {V}ols.
              {I}, {II}, {III}},
    SERIES = {Colloq. Math. Soc. J\'{a}nos Bolyai},
    VOLUME = {10},
     PAGES = {559--583},
 PUBLISHER = {North-Holland, Amsterdam-London},
      YEAR = {1975},
   MRCLASS = {05C15},
  MRNUMBER = {382048},
MRREVIEWER = {Vaclav\ Chv\'{a}tal},
}

@article {CF19,
    AUTHOR = {Conlon, David and Fox, Jacob},
     TITLE = {Lines in {E}uclidean {R}amsey theory},
   JOURNAL = {Discrete Comput. Geom.},
  FJOURNAL = {Discrete \& Computational Geometry. An International Journal
              of Mathematics and Computer Science},
    VOLUME = {61},
      YEAR = {2019},
    NUMBER = {1},
     PAGES = {218--225},
   MRCLASS = {05D10 (52C10)},
  MRNUMBER = {3925553},
MRREVIEWER = {Mikl\'{o}s\ B\'{o}na},
%       DOI = {10.1007/s00454-018-9980-5},
%       URL = {https://doi.org/10.1007/s00454-018-9980-5},
}

@article {CW,
    AUTHOR = {Conlon, David and Wu, Yu-Han},
     TITLE = {More on lines in {E}uclidean {R}amsey theory},
   JOURNAL = {C. R. Math. Acad. Sci. Paris},
  FJOURNAL = {Comptes Rendus Math\'{e}matique. Acad\'{e}mie des Sciences.
              Paris},
    VOLUME = {361},
      YEAR = {2023},
     PAGES = {897--901},
   MRCLASS = {05D10 (52C10)},
  MRNUMBER = {4616159},
MRREVIEWER = {N.\ Hindman},
%       DOI = {10.5802/crmath.452},
%       URL = {https://doi.org/10.5802/crmath.452},
}

@article {LRW,
    AUTHOR = {Leader, Imre and Russell, Paul A. and Walters, Mark},
     TITLE = {Transitive sets in {E}uclidean {R}amsey theory},
   JOURNAL = {J. Combin. Theory Ser. A},
  FJOURNAL = {Journal of Combinatorial Theory. Series A},
    VOLUME = {119},
      YEAR = {2012},
    NUMBER = {2},
     PAGES = {382--396},
   MRCLASS = {05D10 (52C10)},
  MRNUMBER = {2860600},
MRREVIEWER = {Peter\ D.\ Johnson, Jr.},
%       DOI = {10.1016/j.jcta.2011.09.005},
%       URL = {https://doi.org/10.1016/j.jcta.2011.09.005},
}

@article {LRW2,
    AUTHOR = {Leader, Imre and Russell, Paul A. and Walters, Mark},
     TITLE = {Transitive sets and cyclic quadrilaterals},
   JOURNAL = {J. Comb.},
  FJOURNAL = {Journal of Combinatorics},
    VOLUME = {2},
      YEAR = {2011},
    NUMBER = {3},
     PAGES = {457--462},
   MRCLASS = {52C10 (05D10)},
  MRNUMBER = {2913203},
MRREVIEWER = {David\ Conlon},
%       DOI = {10.4310/JOC.2011.v2.n3.a6},
%       URL = {https://doi.org/10.4310/JOC.2011.v2.n3.a6},
}

@article {FR,
    AUTHOR = {Frankl, P. and R\"{o}dl, V.},
     TITLE = {A partition property of simplices in {E}uclidean space},
   JOURNAL = {J. Amer. Math. Soc.},
  FJOURNAL = {Journal of the American Mathematical Society},
    VOLUME = {3},
      YEAR = {1990},
    NUMBER = {1},
     PAGES = {1--7},
   MRCLASS = {52A37 (05A99)},
  MRNUMBER = {1020148},
MRREVIEWER = {R.\ L.\ Graham},
%       DOI = {10.1090/S0894-0347-1990-1020148-2},
%       URL = {https://doi.org/10.1090/S0894-0347-1990-1020148-2},
}

@article {K91,
    AUTHOR = {K\v{r}\'{\i}\v{z}, Igor},
     TITLE = {Permutation groups in {E}uclidean {R}amsey theory},
   JOURNAL = {Proc. Amer. Math. Soc.},
  FJOURNAL = {Proceedings of the American Mathematical Society},
    VOLUME = {112},
      YEAR = {1991},
    NUMBER = {3},
     PAGES = {899--907},
   MRCLASS = {05D10 (20B25)},
  MRNUMBER = {1065087},
MRREVIEWER = {N.\ Hindman},
%       DOI = {10.2307/2048715},
%       URL = {https://doi.org/10.2307/2048715},
}

@article {Sz01,
    AUTHOR = {Szlam, Arthur D.},
     TITLE = {Monochromatic translates of configurations in the plane},
   JOURNAL = {J. Combin. Theory Ser. A},
  FJOURNAL = {Journal of Combinatorial Theory. Series A},
    VOLUME = {93},
      YEAR = {2001},
    NUMBER = {1},
     PAGES = {173--176},
   MRCLASS = {05D10 (52C10)},
  MRNUMBER = {1807118},
MRREVIEWER = {Martin\ Klazar},
%       DOI = {10.1006/jcta.2000.3065},
%       URL = {https://doi.org/10.1006/jcta.2000.3065},
}

@unpublished {AT,
    AUTHOR = {Arman, Andrii and Tsaturian, Sergei},
     TITLE = {Equally spaced collinear points in {E}uclidean {R}amsey theory},
      NOTE = {Preprint available at arXiv:1705.04640 [math.CO]}
}

@unpublished {FT,
    AUTHOR = {F\"uhrer, Jakob and T\'oth, G\'eza},
     TITLE = {Progressions in {E}uclidean {R}amsey theory},
      NOTE = {Preprint available at arXiv:2402.12567 [math.CO]}
}

@unpublished {CMY,
    AUTHOR = {Currier, Gabriel and Moore, Kenneth and Yip, Chi Hoi},
     TITLE = {Avoiding short progressions in {E}uclidean {R}amsey theory},
      NOTE = {Preprint available at arXiv:2404.19233 [math.CO]}
}

@book {Kok,
    AUTHOR = {Koksma, J. F.},
     TITLE = {Some theorems on {D}iophantine inequalities},
    SERIES = {Scriptum},
    VOLUME = {no. 5},
 PUBLISHER = {Math. Centrum, Amsterdam},
      YEAR = {1950},
     PAGES = {i+51},
   MRCLASS = {10.0X},
  MRNUMBER = {38379},
MRREVIEWER = {Paul\ T.\ Bateman},
}

@article {Weyl,
    AUTHOR = {Weyl, Hermann},
     TITLE = {\"{U}ber die {G}leichverteilung von {Z}ahlen mod. {E}ins},
   JOURNAL = {Math. Ann.},
  FJOURNAL = {Mathematische Annalen},
    VOLUME = {77},
      YEAR = {1916},
    NUMBER = {3},
     PAGES = {313--352},
%      ISSN = {0025-5831,1432-1807},
   MRCLASS = {99-04},
  MRNUMBER = {1511862},
%       DOI = {10.1007/BF01475864},
%       URL = {https://doi.org/10.1007/BF01475864},
}

\end{document}